\documentclass[11pt]{article}
\usepackage{amsmath}
\usepackage{amssymb,amsbsy}
\usepackage{graphicx}
\usepackage{dsfont}
\usepackage{natbib}
\usepackage[margin=1in]{geometry}
\usepackage{tikz}
\usetikzlibrary{arrows,decorations.pathmorphing,backgrounds,positioning,fit,petri,mindmap,shapes.geometric,decorations.pathreplacing}

\usepackage{pdfsync}
\usepackage{hyperref}

\allowdisplaybreaks[1]

\newtheorem{theorem}{Theorem}
\newtheorem{proposition}{Proposition}
\newtheorem{lemma}{Lemma}

\renewcommand{\phi}{\varphi}

\renewcommand{\P}{\mathbb{P}}
\newcommand{\E}{\mathbb{E}}
\newcommand{\N}{\mathbb{N}}
\newcommand{\R}{\mathbb{R}}

\newcommand{\cL}{\mathcal{L}}

\newcommand{\cT}{\mathcal{T}}

\def\ds1{\mathds{1}}
\renewcommand{\epsilon}{\varepsilon}

\newlength{\minipagewidth}
\newcommand{\beq}{\begin{equation}}
\newcommand{\eeq}{\end{equation}}

\newcommand{\beqa}{\begin{eqnarray}}
\newcommand{\eeqa}{\end{eqnarray}}

\newcommand{\beqan}{\begin{eqnarray*}}
\newcommand{\eeqan}{\end{eqnarray*}}

\def\ba#1\ea{\begin{align*}#1\end{align*}} 
\def\banum#1\eanum{\begin{align}#1\end{align}} 

\newcommand{\Ga}{\mathrm{Ga}}
\newcommand{\PA}{\mathrm{PA}}
\newcommand{\UA}{\mathrm{UA}}

\newcommand{\BlackBox}{\rule{1.5ex}{1.5ex}}  
\newenvironment{proof}{\par\noindent{\bf Proof\ }}{\hfill\BlackBox\\[2mm]}

\begin{document}
\title{Finding Adam in random growing trees}
\author{
	S{\'e}bastien Bubeck
	\thanks{Microsoft Research and Princeton University; \texttt{sebubeck@microsoft.com}.}
	\and
	Luc Devroye
	\thanks{McGill University; \texttt{lucdevroye@gmail.com}.}
	\and
	G\'abor Lugosi
	\thanks{ICREA and Pompeu Fabra University; \texttt{gabor.lugosi@upf.edu}. GL acknowledges support by the Spanish Ministry of Science and Technology grant MTM2012-37195.}
}
\date{\today}

\maketitle

\begin{abstract}
We investigate algorithms to find the first vertex in large trees generated by either the uniform attachment or preferential attachment model. We require the algorithm to output a set of $K$ vertices, such that, with probability at least $1-\epsilon$, the first vertex is in this set. We show that for any $\epsilon$, there exist such algorithms with $K$ independent of the size of the input tree. Moreover, we provide almost tight bounds for the best value of $K$ as a function of $\epsilon$. In the uniform attachment case we show that the optimal $K$ is subpolynomial in $1/\epsilon$, and that it has to be at least superpolylogarithmic. On the other hand, the preferential attachment case is exponentially harder, as we prove that the best $K$ is polynomial in $1/\epsilon$. We conclude the paper with several open problems.
\end{abstract}

\section{Introduction}
We consider one of the simplest models of a randomly growing graph: starting from a single node (referred to as the {\em root}), each arriving new node connects uniformly at random to one of the existing nodes. 
We are interested in the following question: given a large tree generated from this uniform attachment model, is it possible to find a small set of vertices for which we can certify with high probability that the root is in this set? Possible applications include finding the center of an epidemic, or the origin of a rumor. 

In this paper we study {\em root-finding algorithms}: given a target accuracy $\epsilon \in (0,1)$ and a tree $T$ of size $n$ (for some $n \in \N$), a root-finding algorithm outputs a set $H(T,\epsilon)$ of $K(\epsilon)$ vertices, such that, with probability at least $1-\epsilon$ (with respect to the random generation of $T$ from the uniform attachment model), the root is in $H(T,\epsilon)$. An important aspect of the definition is that the size of the output set is allowed to depend on $\epsilon$, but not on the size $n$ of the input tree. Thus it is not obvious that root-finding algorithms exist at all. For instance a naive guess would be to output vertices of large degrees (indeed, the older a vertex, the larger is its expected degree), but for this to be correct with some constant probability one needs to output a logarithmic in $n$ number of vertices. One of the main contributions of this paper is to show that root-finding algorithms indeed exist. Furthermore, we almost tightly characterize the best possible value for $K(\epsilon)$, by showing that it can be subpolynomial in $1/\epsilon$, and that it has to be superpolylogarithmic. More precisely the core of our contribution can be summarized by the following theorem.

\begin{theorem} \label{th:mainua}
There exist constants $c, c'>0$ such that the following holds true in the uniform attachment model. Any root-finding algorithm must satisfy $K(\epsilon) \geq \exp\left( c \sqrt{\log(1/\epsilon)}\right)$. Furthermore, there exists a polynomial time root-finding algorithm with $K(\epsilon) \leq \exp\left(c' \frac{\log(1/\epsilon)}{\log\log(1/\epsilon)}\right)$.
\end{theorem}

We also investigate the existence of root-finding algorithms for the preferential attachment model, in which each arriving new node connects to an existing node with probability proportional to its degree. In this model the existence of root-finding algorithm is much less surprising, and, in fact, we show that considering vertices of large degrees works here. More interestingly, we prove that the preferential attachment model is {\em exponentially} more difficult than the uniform attachment model, in the sense that $K(\epsilon)$ has to be at least polynomial in $1/\epsilon$ (while it can be subpolynomial in uniform attachment). More precisely, we prove the following theorem.

\begin{theorem} \label{th:mainpa}
There exist constants $c, c'>0$ such that the following holds true in the preferential attachment model. Any root-finding algorithm must satisfy $K(\epsilon) \geq \frac{c}{\epsilon}$. Furthermore, there exists a polynomial time root-finding algorithm with $K(\epsilon) \leq c' \frac{\log^2(1/\epsilon)}{\epsilon^4}$.
\end{theorem}

\subsection{Related work}
There is a vast and rapidly growing literature on both uniform attachment and preferential attachment models. However, to the best of our knowledge, root-finding algorithms have not yet been investigated. On the other hand, a {\em local} variant of root-finding algorithms is studied for the preferential attachment model in \cite{BK10,borgs2012power,FP14}. The restriction to algorithms with only a local access to the graph make the setting quite different from ours. In this context it is proved that the algorithm only has to visit a polylogarithmic (in $n$) number of vertices, which has to be contrasted with our condition that a root-finding algorithm has to output a set of size independent of $n$. We also note that in \cite{SZ11} the authors study root-finding algorithms that are restricted to output a single vertex (that is $K=1$). They are interested in the attainable probability of correctness for a model which can be viewed as the uniform attachment on a {\em background graph}. Interestingly the (inverse) likelihood function for this model, which is called the {\em rumor centrality} in \cite{SZ11}, appears as a relaxation of the likelihood for our model (see Section \ref{sec:relaxationua} for details).

Another recent line of work intimately related to the question studied here is \cite{bubeck2014influencePAseed,curien2014scaling,BEMR14}. In these papers the uniform attachment and preferential attachment models are initialized with some finite seed tree. It is then proved that different seeds lead to different distributions, even in the limit as $n$ goes to infinity. In other words, the uniform attachment and preferential attachment trees are strongly influenced by their state after a finite number of steps, which gives hope that root-finding algorithms indeed exist\footnote{However, chronologically, the work presented here was done before \cite{BEMR14} which proves the influence of the seed in the uniform attachment model.}.

\subsection{Content of the paper}
We start by presenting a simple root-finding algorithm for the uniform attachment model in Section \ref{sec:simpleua}. The idea is to rank vertices according to the size of their {\em largest subtree}. Using basic P{\'o}lya urn results, we show that by taking the smallest $K=\frac{\log(1/\epsilon)}{\epsilon}$ vertices (according to their largest subtree size), one obtains a root-finding algorithm. 

We prove the impossibility result described in Theorem \ref{th:mainua} in Section \ref{sec:mleua}. Using a basic combinatorial argument we show that the optimal estimator
for the root---the maximum likelihood estimator (MLE)---can be computed in polynomial time, and we use well-known results about the uniform attachment tree (such as the behavior of its height) to exhibit the limits of the MLE (and thus of any root-finding algorithm).

In Section \ref{sec:relaxationua} we observe that the root-finding algorithm studied in Section \ref{sec:simpleua} can be viewed as a relaxation of the MLE. We then propose a ``tighter'' relaxation which corresponds to the subpolynomial root-finding algorithm mentioned in Theorem \ref{th:mainua}. The analysis of this algorithm is the most technical part of the paper. It relies on a sharp concentration inequality for sums of Gamma random variables, as well as a beautiful result of Hardy and Ramanujan on the number of partitions of an integer, \cite{HR18}.

Observe that, since the MLE can be computed in polynomial-time, the relaxations of Section \ref{sec:simpleua} and Section \ref{sec:relaxationua} are mainly introduced to simplify the analysis. We note however that these relaxations may still have computational advantages in practice (though the worst-case complexity is of the same order of magnitude).

Finally the preferential attachment model is analyzed in Section \ref{sec:pa}, where Theorem \ref{th:mainpa} is proven. We conclude the paper with several open problems in Section \ref{sec:op}.

\subsection{Notation} \label{sec:notation}
For a labeled tree $T$ we denote by $T^{\circ}$ the isomorphism class of $T$. In other words $T^{\circ}$ is an unlabeled copy of $T$.  For notational convenience we denote vertices from $T^{\circ}$ using the labeling of $T$ (formally one would need to chose an arbitrary labeling of $T^{\circ}$, and to compare vertices of $T^{\circ}$ to those of $T$ one would need to introduce the corresponding isomorphism). We denote by $V(T)$ the vertex set of a labeled tree, and again with a slight abuse of notation we extend this notation to unlabeled trees. The degree of a vertex $v$ is denoted by $d_T(v)$, or simply by $d(v)$ if the underlying tree is clear from the context. An increasing labeling for a tree $T$ with $V(T)=[n]$ is such that any path away from vertex $1$ has increasing labels. A recursive tree is a tree labeled with an increasing labeling. A rooted tree (labeled or unlabeled) is denoted by $(T,u)$, with $u \in V(T)$. We sometimes denote a rooted tree simply as $T$, in which case we denote the root as $\emptyset$. In a recursive tree $T$ it is understood that $\emptyset = 1$. In a rooted tree $T$ we denote by $T_{v\downarrow}$ the subtree starting at $v$. In a rooted tree the descendants of a vertex are referred to as its children, and the set of vertices with no children (i.e., the leaves) is denoted by $\cL(T)$. A plane-oriented recursive tree is a recursive tree together with an ordering of the children of each vertex.
\newline

For $\alpha \in \R$ we define the random labeled tree $\cT_{\alpha}(n)$ with vertex set $[n]$ by induction as follows: $\cT_{\alpha}(2)$ is the unique tree on $2$ vertices, and $\cT_{\alpha}(n+1)$ is built from $\cT_{\alpha}(n)$ by adding the vertex $n+1$ and an edge $\{i, n+1\}$, where $i \in [n]$ is chosen at random with probability proportional to $d_{\cT_{\alpha}(n)}(i)^{\alpha}$ (that is, its degree in $\cT_{\alpha}(n)$ raised to the power $\alpha$). We focus on the case $\alpha=0$, which we alternatively denote $\UA(n)$ for {\em uniform attachment model}, and on $\alpha = 1$ which we denote $\PA(n)$ for {\em preferential attachment model}. It is well known that $\UA(n)$ is equivalently described as a uniformly chosen tree among all recursive trees, while $\PA(n)$ can be described as a uniformly chosen tree among all plane-oriented trees, see e.g., \cite{drmota2009random}. For this reason $\UA(n)$ is also referred to as the 
uniform random recursive tree, and $\PA(n)$ as the random plane-oriented recursive tree.
\newline

We can now formalize the problem introduced at the beginning of the introduction. Let $\alpha \in \{0,1\}$, $\epsilon \in (0,1)$, and $K \in \N$. We are interested in mappings $H$ from unlabeled trees to subsets of $K$ vertices with the property that 
\begin{equation} \label{eq:goal}
\underset{n \to +\infty}{\mathrm{liminf}} \ \P(1 \in H(\cT_{\alpha}(n)^{\circ})) \geq 1-\epsilon
\end{equation}
In most instances the input tree will be clear from the context, and thus we often write $H$ instead of $H(\cT_{\alpha}(n)^{\circ})$.

\section{A simple root-finding algorithm} \label{sec:simpleua}
For a tree $T$ we introduce the function $\psi_T : V(T) \rightarrow \N$, defined by
\begin{equation} \label{eq:psi}
\psi_T(u) = \max_{v \in V(T) \setminus \{u\}} |(T,u)_{v \downarrow}| .
\end{equation}
In words, viewing $u$ as the root of $T$, $\psi_T(u)$ returns the size of the largest subtree starting at a child of $u$. We denote by $H_{\psi}$ the mapping which returns the set of $K$ vertices with smallest $\psi$ values (ties are broken arbitrarily). The following theorem shows that $H_{\psi}$ is a root-finding algorithm for the uniform attachment model. We also prove in Section \ref{sec:pa} that it is a root-finding algorithm for the preferential attachment.

\begin{theorem} \label{th:psi}
Let $K \geq 2.5 \frac{ \log(1/ \epsilon)}{\epsilon}$. One has $\underset{n \to +\infty}{\mathrm{liminf}} \ \P(1 \in H_{\psi}(\UA(n)^{\circ})) \geq 1 - \frac{4 \epsilon}{1-\epsilon}$. 
\end{theorem}
We observe that the theorem is optimal up to a logarithmic factor. Indeed, leaves clearly
maximize the value of $\psi$ and one can easily see that, with probability at least $\Omega(1/K)$, the true root is a leaf in $\UA(K)$. Nonetheless, perhaps surprisingly, there is an exponentially better root-finding algorithm than $H_{\psi}$ for the uniform attachment model.

\begin{proof} 
In this proof vertices are labeled by chronological order. We also introduce a notation which will prove useful; for $1 \leq i \leq k$ we denote by $T_{i,k}$ the tree containing vertex $i$ in the forest obtained by removing in $\UA(n)$ all edges between vertices $\{1,\hdots, k\}$. In particular the vector $(|T_{1,k}|, \hdots, |T_{k,k}|)$ follows a standard P\'olya urn with $k$ colors, and thus, using a classical result, 
$$\frac{1}{n} (|T_{1,k}|, \hdots, |T_{k,k}|)$$
converges, in distribution, to a Dirichlet distribution with parameters $(1, \ldots, 1).$
Now observe first that
$$\P(1 \not\in H_{\psi}) \leq \P(\exists i > K : \psi(i) \leq \psi(1)) \leq \P(\psi(1) \geq (1-\epsilon) n) + \P(\exists i > K : \psi(i) \leq (1-\epsilon) n) .$$
Clearly
$$\psi(1) \leq \max(|T_{1,2}|, |T_{2,2}| ),$$
and thus since $|T_{1,2}|/n$ and $|T_{2,2}|/n$ are identically distributed and converge in distribution to a uniform random variable in $[0,1]$,
$$\limsup_{n \to +\infty} \P(\psi(1) \geq (1-\epsilon) n) \leq 2 \lim_{n \to +\infty} \P(|T_{1,2}| \geq (1-\epsilon) n ) = 2 \epsilon .$$
On the other hand, for any $i > K$,
$$\psi(i) \geq \min_{1 \leq k \leq K} \sum_{j=1, j \neq k}^K |T_{j,K}| ,$$
and $\frac{1}{n} \sum_{j=1, j \neq k}^K |T_{j,K}|$ converges, in distribution, to the $\text{Beta}(K-1,1)$ distribution, which implies
\begin{eqnarray*}
\limsup_{n \to +\infty} \P(\exists i > K : \psi(i) \leq (1-\epsilon) n) & \leq & \lim_{n \to +\infty} \P\left(\exists 1 \leq k \leq K : \sum_{j=1, j \neq k}^K |T_{j,K}| \leq (1-\epsilon) n \right) \\
& \leq & K (1-\epsilon)^{K-1} .
\end{eqnarray*}
Thus we proved 
$$\limsup_{n \to +\infty} \P(1 \not\in H_{\psi}) \leq 2 \epsilon + K (1-\epsilon)^{K-1} ,$$
which clearly concludes the proof.
\end{proof}

\section{Maximum likelihood estimator for $\UA(n)$} \label{sec:mleua}
For a rooted tree $T$ and a vertex $v$, we define $\mathrm{Aut}(v, T)$ as follows. Let $T_1, \hdots, T_{k}$ be the subtrees of $T$ rooted at the children of $v$ (in particular $k \in \{d(v)-1, d(v)\}$). Let $S_1, \hdots, S_{L}$ be the different isomorphism classes realized by these rooted subtrees. For $i \in [L]$, let $\ell_i = | \{j \in [k] : T_j^{\circ} = S_i\} |$. Finally we let $\mathrm{Aut}(v, T) := \prod_{i = 1}^{L} \ell_i !$.
\begin{proposition} \label{prop:counting}
Let $T$ be an unlabeled rooted tree, then
$$|\{t \ \text{recursive tree} \ : t^{\circ} = T\}| = \frac{|T|!}{\prod_{v \in V(T) \setminus \cL(T)} \left( |T_{v \downarrow}| \cdot  \mathrm{Aut}(v, T) \right)} .$$
\end{proposition}

\begin{proof}
We prove this result by induction on the number of vertices in $T$, which we denote by $n$ (the formula is clearly true for $n=2$). Let $T_1, \hdots, T_k$ be the subtrees rooted at the children of $\emptyset$ (in particular $k=d(\emptyset)$). An increasing labeling for $T$ is obtained by partitioning $\{2, \hdots, n\}$ into $k$ subsets of sizes $|T_1|, \hdots, |T_k|$, and then labeling correctly each subtree with the corresponding element of the partition. Clearly the number of possibilities for choosing the partition is $\frac{(n-1)!}{\prod_{i=1}^k |T_i|!}$, and given the partition the number of correct labelings is $\prod_{i=1}^k |\{t \ \text{recursive tree} \ : t^{\circ} = T_i\}|$. In this calculation we have counted multiple times similar recursive trees (of size $n$). Indeed if $T_i$ and $T_j$ are isomorphic then we have considered separately the case where $T_i$ is labeled with $S \subset \{2, \hdots, n\}$ and $T_j$ is labeled with $S'$, and the case where $T_i$ is labeled with $S'$ and $T_j$ is labeled with $S$. In fact we have precisely overcounted by a factor $\mathrm{Aut}(\emptyset, T)$. Thus we obtain the formula
$$|\{t \ \text{recursive tree} \ : t^{\circ} = T\}| = \frac{(n-1)!}{\mathrm{Aut}(\emptyset, T) \prod_{i=1}^k |T_i|!} \prod_{i=1}^k |\{t \ \text{recursive tree} \ : t^{\circ} = T_i\}| ,$$
which easily concludes the proof (by using the induction hypothesis).
\end{proof}
For an unrooted unlabelled tree $T$ let $\overline{\mathrm{Aut}}(u, T)$ be the number of vertices $v$ such that $(T,v)$ is isomorphic to $(T,u)$. Observe that the probability that a vertex $u$ in $T$ is the root (in the $\UA(n)$ model) is equal to the probability of observing the rooted tree $(T,u)$ (which is proportional to the number of recursive trees $t$ such that $t^{\circ} = (T,u)$) divided by $\overline{\mathrm{Aut}}(u, T)$. In particular Proposition \ref{prop:counting} implies that, given an  observation $T$ (i.e., an unlabeled tree on $n$ vertices), the maximum likelihood estimator for the root in the uniform attachment model is the vertex minimizing the function
\begin{equation} \label{eq:zeta}
\zeta_T(u) = \overline{\mathrm{Aut}}(u, T) \prod_{v \in V(T) \setminus \cL((T,u))} \left( |(T,u)_{v \downarrow}| \cdot \mathrm{Aut}(v, (T,u)) \right) .
\end{equation} 
In fact it implies more generally that the optimal strategy to output a set of $K$ vertices is to choose those with the smallest values of $\zeta$. This follows from the fact that, conditionally on the observation $T$, the probability that a set of vertices contains the root is proportional to the sum of the inverse $\zeta$ values for the vertices in this set. We denote the mapping corresponding to this optimal strategy by $H_{\zeta}$. Using this representation for the optimal procedure we prove now the following impossibility result:

\begin{theorem} \label{th:zeta}
There exists $\epsilon_0>0$ such that for all $\epsilon \leq \epsilon_0$, any procedure that satisfies \eqref{eq:goal} in the uniform attachment model must have $K \geq \exp\left( \sqrt{\frac{1}{30} \log \frac1{2\epsilon}}\right)$.
\end{theorem}

\begin{proof}
Let\footnote{For clarity of the presentation we ignore integer rounding issues in this paper.} $K =  \exp\left( \sqrt{\frac{1}{30} \log \frac1{2\epsilon}}\right)$, which we assume to be large enough (that is $\epsilon$ is small enough, see below). 
Since $H_{\zeta}$ is the optimal procedure, one clearly has that $\P(1 \not\in H_{\zeta}(\UA(n)^{\circ}))$ is non-decreasing with $n$ (since any procedure for trees of size $n+1$ can be simulated given a tree of size $n$), and thus to prove the theorem it is enough to show that $\P(1 \not\in H_{\zeta}(\UA(K+1)^{\circ})) > \epsilon$. Equivalently we need to show that there is a set $\cT$ of recursive trees on $K+1$ vertices such $\P(\UA(K+1) \in \cT) > \epsilon$ and for any $T \in \cT$ one has $\zeta_T(1) > \zeta_T(i), \forall i \in \{2, \hdots, K+1\}$.

The set of recursive trees $\cT$ we consider consists of those where (i) vertices $\{1, \hdots, 10 \log(K)\}$ form a path with $1$ being an endpoint, (ii) all vertices in $\{10\log(K) +1, \hdots, K+1\}$ are descendants of $10 \log(K)$, and (iii) the height of the subtree rooted at $10 \log(K)$ is smaller than $4 \log(K)$. Next we verify that this set of recursive trees has probability at least $\frac{1}{2} \exp(- 30 \log^2(K)) = \epsilon$. More precisely the probability that (i) happens is exactly $1/(10 \log(K) - 1)! \geq \exp(- 10 \log^2(K))$; the probability that (iii) happens conditionally on (ii) is at least $1/2$ for $K$ large enough (indeed it is well-known that the height of $\UA(n)$ rescaled by $\log(n)$ converges in probability to $e$, see \cite{Dev87}); and the probability that (ii) happens is equal to
$$\prod_{i=10\log(K)}^K \left(1 - \frac{10 \log(K) - 1}{i} \right) \geq \exp(- 20 \log^2(K)) .$$

We show now that for trees in $\cT$, the root $1$ has the largest $\zeta$ value, thus concluding the proof. First observe that, in general for any tree $T$ and vertices $u, v$, with $(u_1, \hdots, u_k)$ being the unique path with $u_1 = u$ and $u_k = v$,
\begin{align*}
& \zeta_T(u) > \zeta_T(v) \\
& \Leftrightarrow \overline{\mathrm{Aut}}(u, T) \prod_{i=1}^k \bigg(|(T,u)_{u_i \downarrow}| \cdot \mathrm{Aut}(u_i, (T,u)) \bigg) > \overline{\mathrm{Aut}}(v, T) \prod_{i=1}^k \bigg(|(T,v)_{u_{i} \downarrow}| \cdot \mathrm{Aut}(u_i, (T,v)) \bigg) .
\end{align*}
Furthermore it is easy to verify that one always has
$$\mathrm{Aut}(u_i, (T,v)) \leq |(T,v)_{u_{i} \downarrow}| \cdot \mathrm{Aut}(u_i, (T,u)) .$$
Indeed the computation of $\mathrm{Aut}(u_i, (T,v))$ is based on a list of subtrees $T_1, \hdots, T_k$, and only one of those subtrees is modified in the computation of $\mathrm{Aut}(u_i, (T,u))$ (or possibly a subtree is added if $i=1$), which results in a mutiplicative change of at most $k +1 \leq |(T,v)_{u_{i} \downarrow}|$.

Putting together the two displays above, and using the trivial bound $1 \leq \overline{\mathrm{Aut}}(u, T) \leq |T|$ one obtains
\begin{equation} \label{eq:coolone}
\prod_{i=1}^k |(T,u)_{u_i \downarrow}| > |T| \prod_{i=1}^k |(T,v)_{u_{i} \downarrow}|^2 \Rightarrow \zeta_T(u) > \zeta_T(v) .
\end{equation}
Now consider $T \in \cT$, and $v \in \{10 \log(K)+1, \hdots, K+1\}$. Let $(u_1, \hdots, u_k)$ be the unique path with $u_1 = 1$ and $u_k = v$. Clearly
$$\prod_{i=1}^k |(T,v)_{u_i \downarrow}| \leq (K+1)^{4 \log(K)} \cdot (10 \log(K)) ! \leq  (K+1)^{4 \log(K)} \cdot (10 \log(K))^{10 \log(K)},$$
and 
$$\prod_{i=1}^k |(T,u)_{u_i \downarrow}| \geq (K+1 - 10 \log(K))^{10 \log(K)} .$$
Thus using \eqref{eq:coolone} one has that for $K$ large enough, $\zeta_T(1) > \zeta_T(v)$. Furthermore it is obvious that $\zeta_T(1) > \zeta_T(v)$ for $v \in \{2, \hdots, 10 \log(K)\}$. This concludes the proof.
\end{proof}

\section{A subpolynomial root-finding algorithm for $\UA(n)$} \label{sec:relaxationua}
The performance of the maximum likelihood estimate is complicated to analyze
due to the presence of the automorphism numbers in the expression (\ref{eq:zeta}).
In order to circumvent this difficulty, we analyze algorithms that minimize
modified versions of the maximum likelihood criterion. 
The function $\psi$ defined in \eqref{eq:psi} can be viewed as such a ``relaxation'' of the likelihood function $\zeta$ defined in \eqref{eq:zeta}. In this section we analyze a tighter relaxation, which we denote $\phi$ and define as 
\begin{equation} \label{eq:phi}
\phi_T(u) = \prod_{v \in V(T) \setminus \{u\}} |(T,u)_{v \downarrow}| .
\end{equation}
We denote by $H_{\phi}$ the mapping which returns the set of $K$ vertices with smallest $\phi$ values. We also note that $1 / \phi_T(u)$ is referred to as the {\em rumor centrality} of vertex $u$ in \cite{SZ11}.

\begin{theorem} \label{th:phi}
There exist universal constants $a,b>0$ such that if
$K \geq a \exp \left(b\frac{\log(1/\epsilon)}{\log\log(1/\epsilon)}\right)$, 
then one has $\underset{n \to +\infty}{\mathrm{liminf}} \ \P(1 \in H_{\phi}(\UA(n)^{\circ})) \geq 1 - \epsilon$. 
\end{theorem}

Before going into the proof of the above result, we start with a technical lemma on the concentration of sums of Gamma random variables.

\begin{lemma} \label{lem:concgamma}
Let $j_1, \hdots j_{\ell} \in \N$, and $s = \sum_{k=1}^{\ell} k j_k$. For $k \in [\ell]$, let\footnote{We denote $\Ga(a,b)$ for the Gamma distribution with density proportional to $x^{a-1} \exp(-x/b) \mathbf{1}_{\left\{x > 0 \right\}}$.} $X_k \sim \Ga(j_k, k)$, with $X_1, \hdots, X_{\ell}$ being independent. Then for any $t \in (0,s)$,
$$\P\left(\sum_{k=1}^{\ell} X_k < t\right) \leq \exp\left(- \sqrt{\frac{s}{2}} \log \left(\frac{s}{e t}\right)\right) .$$
\end{lemma}

\begin{proof}
Recall that for any $\lambda \geq 0$, one has
$$\E_{X \sim \Ga(a,b)} \exp(- \lambda X) = \frac{1}{(1+\lambda b)^a} .$$
Thus using Chernoff's method one obtains (using also that $x \mapsto \frac{\log(1+x)}{x}$ is non-increasing for the second inequality)
\begin{eqnarray*}
\P\left(\sum_{k=1}^{\ell} X_k < t\right) \leq \exp(\lambda t) \E \exp\left( - \lambda \sum_{k=1}^{\ell} X_k \right)
& = & \exp\left(\lambda t - \sum_{k=1}^{\ell} j_k \log(1+\lambda k)\right) \\
& \leq & \exp\left(\lambda t - \sum_{k=1}^{\ell} \lambda k j_k \frac{\log(1+\lambda \ell)}{\lambda \ell} \right) \\
& = & \exp\left(\lambda t - \frac{s}{\ell} \log(1+\lambda \ell)\right).
\end{eqnarray*}
Taking $\lambda = (s/t - 1) / \ell > 0$ yields
$$\P\left(\sum_{k=1}^{\ell} X_k < t\right) \leq \exp\left( - \frac{s \log \left(\frac{s}{e t}\right)}{\ell} \right).$$
It remains to observe that $s \geq \ell^2/2$ to conclude the proof.
\end{proof}

We also recall Erd\H{o}s' non-asymptotic version of the Hardy-Ramanujan formula on the number of partitions of an integer, \cite{Erd42},
\begin{equation} \label{eq:HR}
\left| \left\{(j_1, \hdots, j_{\ell}) \in \N^{\ell} \ \text{s.t.} \ \ell \in \N, \sum_{k=1}^{\ell} k j_k = s \right\}\right| \leq \exp \left(\pi \sqrt{\frac{2}{3} s} \right) .
\end{equation}

\begin{proof}
We decompose the proof into four steps. Denote $T=\UA(n)^{\circ}$, and let $S \geq 1$ be a value appropriately chosen later.

\medskip
\noindent
\textbf{Step 1.} We first introduce an alternative labeling of the vertices in $\UA(n)$. We define this labeling by induction as follows. The root is labeled with the empty set $\emptyset$. Node $(j_1, \hdots, j_{\ell}) \in \N^{\ell}$ with $\ell \in \N$ is defined to be the $j_{\ell}^{th}$ children (in birth order) of node $(j_1, \hdots, j_{\ell-1})$. Thus instead of labeling vertices with elements of $\N$, they are labeled with elements of $\N^* := \cup_{\ell=0}^{\infty} \N^{\ell}$. For any vertex $v \in \N^*$ we define $\ell(v)$ to be such that $v \in \N^{\ell(v)}$ (in other words $\ell(v)$ is the depth of $v$), and $s(v) = \sum_{k=1}^{\ell(v)} (\ell(v)+1-k) j_k$.

We observe the following important property: for any vertex $v$ such that $s(v) > 3 S$, one has either
\begin{itemize}
\item[(i)] there exists $u$ such that $s(u) \in (S, 3S]$ and $v \in (T, \emptyset)_{u \downarrow}$, or,
\item[(ii)] there exists $u$ such that $s(u) \leq S$ and $v \in (T, \emptyset)_{(u,j) \downarrow}$ for some $j > S$.
\end{itemize}
We prove this property by induction on the depth $\ell(v)$ of $v$. For $\ell(v)=1$, (ii) is clearly true with $u=\emptyset$. For $\ell(v) >1$, let $u$ be the parent of $v$. We now have three cases:
\begin{itemize}
\item[(a)] If $s(u) > 3 S$, then one can apply the induction hypothesis on $u$.
\item[(b)] If $s(u) \in (S, 3 S]$, then (i) is true.
\item[(c)] Finally if $s(u) \leq S$, then $v$ is the $j^{th}$ children of $u$ with $j > S$ (this uses the fact that $s((u,j)) \leq j + 2 S$, and $s(v) > 3S$), and thus (ii) is true.
\end{itemize}
This concludes the proof that either (i) or (ii) is always true. Also note that $\phi$ satisfies for any vertex $w$ on a path from $u$ to $v$, $\phi(w) \leq \max(\phi(u), \phi(v))$. Putting these facts together we proved that
\begin{align}
& \P\left(\exists \ v : s(v) > 3 S \ \text{and} \ \phi(v) \leq \phi(1) \right) \notag \\
& \leq \P\left(\exists \ v : s(v) \in (S, 3 S] \ \text{and} \ \phi(v) \leq \phi(1) \right) + \P\left( \exists \ v, j : s(v) \leq S, j > S, \ \text{and} \ \phi((v,j)) \leq \phi(1) \right). \label{eq:finiteunion}
\end{align}

\medskip
\noindent
\textbf{Step 2.} We show here that, after a union bound on $v$, the second term in \eqref{eq:finiteunion} is bounded by the first term. First observe that, for $v=(j_1, \hdots, j_{\ell})$, one has
\begin{equation} \label{eq:eqphi1}
\phi(v) \leq \phi(1) \Leftrightarrow \prod_{i=1}^{\ell(v)} |(T,\emptyset)_{(j_1, \hdots, j_i) \downarrow}| \geq \prod_{i=0}^{\ell(v)-1} |(T,v)_{(j_1, \hdots, j_i)  \downarrow}| = \prod_{i=1}^{\ell(v)} \left(n - |(T,\emptyset)_{(j_1, \hdots, j_i)  \downarrow}|\right).
\end{equation}
In particular, this implies that
\begin{align*}
& \exists \ j > S : \phi((v,j)) \leq \phi(1) \\
& \Rightarrow \prod_{i=1}^{\ell(v)} |(T,\emptyset)_{(j_1, \hdots, j_i) \downarrow}| \cdot \left(\sum_{j=S+1}^{+\infty} |(T,\emptyset)_{(v,j) \downarrow}| \right) \geq \prod_{i=1}^{\ell(v)} \left(n - |(T,\emptyset)_{(j_1, \hdots, j_i)  \downarrow}|\right) \cdot \left(n - \sum_{j=S+1}^{+\infty} |(T,\emptyset)_{(v,j) \downarrow}| \right).
\end{align*}
Now it is easy to see that the random variable $\sum_{j=S+1}^{+\infty} |(T,\emptyset)_{(v,j) \downarrow}|$ is stochastically dominated by $|(T,\emptyset)_{(v,S) \downarrow}|$, which, together with the two above displays, show that 
$$\P(\exists \ j > S : \phi((v,j)) \leq \phi(1)) \leq \P(\phi((v,S)) \leq \phi(1) ).$$
In particular, putting this together with \eqref{eq:finiteunion} and an union bound, one obtains:
\begin{equation} \label{eq:eqphi3}
\P\left(\exists \ v : s(v) > 3 S \ \text{and} \ \phi(v) \leq \phi(1) \right) \leq 2 \sum_{v : s(v) \in (S, 3S]} \P\left(\phi(v) \leq \phi(1) \right) .
\end{equation}

\medskip
\noindent
\textbf{Step 3.}
This is the main step of the proof, where we show that for any vertex $v$ with $s(v) \geq 10^{10}$,
\begin{equation} \label{eq:eqphi2}
\limsup_{n \to +\infty} \P(\phi(v) \leq \phi(1)) \leq 7 \exp\left(- 0.21 \sqrt{s(v)} \log(s(v)) \right) .
\end{equation}
Observe that the random vector $\left(\frac{1}{n}|(T,\emptyset)_{(j_1, \hdots, j_i)\downarrow}| \right)_{i=1, \hdots, \ell(v)}$ evolves according to a "pile" of P{\'o}lya urns, and thus using a standard argument it converges in distribution to the random vector $\left( \prod_{k=1}^i U_{j_k, k}\right)_{i=1, \hdots, \ell(v)}$,
where $U_{j,1}, U_{j,2}, \hdots$ are (independent) products of $j$ independent uniform random variables in $[0,1]$. In particular, by \eqref{eq:eqphi1} this gives
\begin{align*}
& \limsup_{n \to +\infty} \P(\phi(v) \leq \phi(1)) \\
& =  \P \left(\prod_{i=1}^{\ell(v)} \prod_{k=1}^i U_{j_k, k} \geq \prod_{i=1}^{\ell(v)} \left(1 - \prod_{k=1}^i U_{j_k, k}\right) \right) \\
& = \P \left(\prod_{i=1}^{\ell(v)} U_{j_i, i}^{\ell(v)+1-i} \geq \prod_{i=1}^{\ell(v)} \left(1 - \prod_{k=1}^i U_{j_k, k}\right) \right) \\
& \leq  \P \left(\prod_{i=1}^{\ell(v)} U_{j_i, i}^{\ell(v)+1-i} \geq \exp\left(- \frac{s(v)}{R}\right) \right) + \P\left(\prod_{i=1}^{\ell(v)} \left(1 - \prod_{k=1}^i U_{j_k, k}\right) \leq \exp\left(- \frac{s(v)}{R}\right) \right).
\end{align*}
where $R>0$ will be appropriately chosen later. Using Lemma \ref{lem:concgamma}, one directly obtains
$$ \P \left(\prod_{i=1}^{\ell(v)} U_{j_i, i}^{\ell(v)+1-i} \geq \exp\left(- \frac{s(v)}{R}\right) \right) \leq  \exp\left(- \sqrt{\frac{s(v)}{2}} \log \left(\frac{R}{e}\right)\right) .$$
Furthermore, using that $1-\exp(-x) \geq \frac{1}{2} \min(x,1)$ for any $x \geq 0$, one can check that 
$$\prod_{i=1}^{\ell(v)} \left(1- \prod_{k=1}^i U_{j_k, k}\right) \geq\frac{1}{2^{\ell(v)}} \prod_{i=1}^{\ell(v)} \min\left( \sum_{k=1}^i \log(1/U_{j_k,k}) , 1 \right) \geq \frac{1}{2^{\ell(v)}} X, $$
where $X$ is equal in distribution to $\prod_{i=1}^{+\infty} \min\left(\sum_{k=1}^i \log(1/U_{1,k}) , 1 \right)$. Thus, using Lemma \ref{lem:lucisaslave} in the Appendix, one gets,
$$\P \left(\prod_{i=1}^{\ell(v)} \left(1 - \prod_{k=1}^i U_{j_k, k}\right) \leq \exp\left(- \frac{s(v)}{R}\right) \right) \leq 6 \cdot 2^{\frac{\ell(v)}{4}} \exp\left(- \frac{s(v)}{4R}\right) \leq 6 \exp\left(\frac{\sqrt{s(v)}}{4} - \frac{s(v)}{4R}\right) ,$$
where we used $s(v) \geq \ell(v)/2$ for the second inequality. Taking $R= e \cdot s(v)^{0.3}$ easily concludes the proof of \eqref{eq:eqphi2}.

\medskip
\noindent
\textbf{Step 4.} Putting together \eqref{eq:eqphi3} and \eqref{eq:eqphi2} one obtains that for $S \geq 10^{10}$,
$$\limsup_{n \to +\infty} \P\left(\exists \ v : s(v) > 3 S \ \text{and} \ \phi(v) \leq \phi(1) \right) \leq 14 \cdot |\{v : s(v) \leq 3 S\}| \cdot \exp\left(- 0.21 \sqrt{S} \log(S) \right) .$$
Furthermore by \eqref{eq:HR} one clearly has $|\{v : s(v) \leq 3 S\}| \leq 3 S \exp(\pi \sqrt{2 S})$, and thus for $S \geq 10^{10}$ one can check that
$$\limsup_{n \to +\infty} \P\left(\exists \ v : s(v) > 3 S \ \text{and} \ \phi(v) \leq \phi(1) \right) \leq \exp\left(- \frac{1}{100} \sqrt{S} \log(S) \right) ,$$
which easily concludes the proof.

\end{proof}

\section{Root-finding algorithms for $\PA(n)$} \label{sec:pa}
In this section we investigate the preferential attachment model. We first analyze the simple root-finding algorithm defined in Section \ref{sec:simpleua}.
\begin{theorem} \label{th:psipa}
Let $K \geq C \frac{ \log^2(1/ \epsilon)}{\epsilon^4}$ for some numerical constant $C>0$. One has $\underset{n \to +\infty}{\mathrm{liminf}} \ \P(1 \in H_{\psi}(\PA(n)^{\circ})) \geq 1 - \epsilon$. 
\end{theorem}
On the contrary to the situation for uniform attachment it is not clear if the above theorem gives the correct order of magnitude in $\epsilon$: With probability at least $\Omega(1/\sqrt{K})$, the root is a leaf in $\PA(K)$, and thus for $H_{\psi}$ to be correct with probability at least $1-\epsilon$ one needs $K = \Omega(1/\epsilon^2)$. In other words there is a quadratic gap between this lower bound and the upper bound given by the above result. We also note that a similar bound to the one given by Theorem \ref{th:psipa} could be obtained by simply considering the vertices with largest degree (in fact it could even give a better polynomial dependency in $\epsilon$). However the proof of this latter statement is more involved technically, because degrees behave as triangular P{\'o}lya urns which have more complicated limiting distributions, see, e.g., \cite{Jan06}. On the contrary, subtree sizes are P{\'o}lya urns with diagonal replacement matrix, which allows for the (relative) simplicity of the following proof.

\begin{proof} 
Similarly to the proof of Theorem \ref{th:psi} we label vertices by chronological order, and for $1 \leq i \leq k$ we denote by $T_{i,k}$ the tree containing vertex $i$ in the forest obtained by removing in $\PA(n)$ all edges between vertices $\{1,\hdots, k\}$. Here the vector $2 (|T_{1,k}|, \hdots, |T_{k,k}|)$ follows a P\'olya urn with $k$ colors and replacement matrix $2 I_k$ (where $I_k$ is the $k\times k$ identity matrix), and thus using again a classical result one has the following convergence in distribution, conditionally on $\PA(k)$,
$$\frac{1}{n} (|T_{1,k}|, \hdots, |T_{k,k}|) \xrightarrow[n \to +\infty]{} \text{Dir}\left( \frac{d_{\PA(k)}(1)}{2}, \hdots, \frac{d_{\PA(k)}(k)}{2}\right) .$$
Let $\eta \in (0,1)$ and observe that
$$\P(1 \not\in H_{\psi}) \leq \P(\exists i > K : \psi(i) \leq \psi(1)) \leq \P(\psi(1) \geq (1-\eta) n) + \P(\exists i > K : \psi(i) \leq (1-\eta) n) .$$
Using that $\psi(1) \leq \max(|T_{1,2}|, |T_{2,2}| )$,
and since $|T_{1,2}|/n$ and $|T_{2,2}|/n$ are identically distributed and converge in distribution to a $\text{Beta}(1/2,1/2)$, one has
$$\limsup_{n \to +\infty} \P(\psi(1) \geq (1-\eta) n) \leq 2 \lim_{n \to +\infty} \P(|T_{1,2}| \geq (1-\eta) n ) = \frac{2}{\pi} \mathrm{arcsin}(\sqrt{\eta}) \leq \sqrt{\eta} .$$
On the other hand, for any $i > K$, $\psi(i) \geq \min_{1 \leq k \leq K} \sum_{j=1, j \neq k}^K |T_{j,K}|$,
and $\frac{1}{n} \sum_{j=1, j \neq k}^K |T_{j,K}|$ is stochastically lower bounded by $\frac{1}{n} \sum_{j=2}^K |T_{j,K}|$ which converges in distribution to a 
$$\text{Beta}\left(K-1 - \frac{d_{\PA(K)}(1)}{2} ,\frac{d_{\PA(K)}(1)}{2}\right) .$$
Thus we have
\begin{eqnarray*}
\limsup_{n \to +\infty} \P(\exists i > K : \psi(i) \leq (1-\eta) n) & \leq & \lim_{n \to +\infty} \P\left(\exists 1 \leq k \leq K : \sum_{j=1, j \neq k}^K |T_{j,K}| \leq (1-\eta) n \right) \\
& \leq & K \ \P\left( \text{Beta}\left(K-1 - \frac{d_{\PA(K)}(1)}{2} ,\frac{d_{\PA(K)}(1)}{2}\right)  \leq 1-\eta\right) .
\end{eqnarray*}
Using properties of Beta distributions and triangular P\'olya urns (in particular [Example 3.1, \cite{Jan06}]) one can show that this last term is smaller than $\eta$ for $K \geq C \frac{ \log(1/ \eta)}{\eta^2}$. The proof is thus concluded by taking $\eta = \epsilon^2$. 
\end{proof}

Next we prove a general impossibility result. On the contrary to the proof of Theorem \ref{th:zeta}, here we do not use the structure of the maximum likelihood estimator, as a simple symmetry argument suffices. 

\begin{theorem} \label{th:xi}
There exists $c > 0$ such that for $\epsilon \in (0,1)$, any procedure that satisfies \eqref{eq:goal} in the preferential attachment model must have $K \geq c / \epsilon$.
\end{theorem}

\begin{proof}
As we observed in the proof of Theorem \ref{th:zeta}, the probability of error for the optimal procedure is non-decreasing with $n$, so it suffices to show that the optimal procedure must have a probability of error of at least $\epsilon$ for some finite $n$. 

We show that there is some finite $n$ for which, with probability at least $2\epsilon$, $1$ is isomorphic to at least $2 c /\epsilon$ vertices in $\PA(n)$. This clearly implies a probability of error of at least $\epsilon$ for any procedure that outputs less than $c/\epsilon$ vertices. To simplify the proof we now use the $\Theta$ notation. First observe that the probability that $1$ is a leaf is $\Theta(1/\sqrt{n})$, and thus for $n = \Theta(1/\epsilon^2)$ this happens with probability at least $\Theta(\epsilon)$. Furthermore it is an easy exercise\footnote{Simply recall that $d_{\PA(n)}(1)/\sqrt{n}$ converges in distribution (this follows from \cite{Jan06}) and at least a constant fraction of the vertices in $\{n/2,\hdots, n\}$ are leaves.} to verify that, conditioned on $1$ being a leaf, with probability at least $\Theta(1)$, $2$ is connected to $\Theta(\sqrt{n}) = \Theta(1/\epsilon)$  leaves, which are then isomorphic to $1$.
\end{proof}
For sake of completeness, we observe that the maximum likelihood estimator in the preferential attachment model is obtained by minimizing the function
\begin{equation} \label{eq:zeta}
\xi_T(u) = \frac{\overline{\mathrm{Aut}}(u,T)}{d_T(u)}{\prod_{v \in V(T) \setminus \cL((T,u))} \left( |(T,u)_{v \downarrow}| \cdot \mathrm{Aut}(v, (T,u)) \right)} .
\end{equation}
This is a consequence of the following observation.
\begin{proposition} \label{prop:countingPA}
Let $T$ be an unlabeled rooted tree with $n$ vertices, then
$$|\{t \ \text{plane-oriented recursive tree} \ : t^{\circ} = T\}| = \frac{n! \cdot d_T(\emptyset)! \cdot \prod_{v \in V(T) \setminus \{\emptyset\}} (d_T(v)-1)!}{\prod_{v \in V(T) \setminus \cL(T)} \left( |T_{v \downarrow}| \cdot  \mathrm{Aut}(v, T) \right)} .$$
\end{proposition}

\begin{proof}
Recall that a plane-oriented recursive tree is a recursive tree together with an ordering of the children of each vertex. For the root there are $d_T(\emptyset)!$ possible orderings of the children, while for any other vertex $v \neq \emptyset$ there are only $(d_T(v)-1)!$ possible orderings. Together with Proposition \ref{prop:counting} this immediately yields the above formula.
\end{proof}

\section{Open problems} \label{sec:op}
\begin{enumerate}
\item Our results leave gaps for both uniform attachment and preferential attachment models. In particular, is it possible to find a procedure with $K \leq \exp(c \sqrt{\log(1/\epsilon)})$ (respectively $K \leq c / \epsilon$) in the uniform attachment (respectively preferential attachment) model?
\item In the seeded models of \cite{bubeck2014influencePAseed, BEMR14}, how large does $K$ need to be to certify that the obtained set of vertices contains the seed with probability at least $1-\epsilon$?
\item In Section \ref{sec:notation} we introduced the more general non-linear preferential attachment model $\cT_{\alpha}(n)$. What can be said for this model?
\item What about growing graphs instead of trees? For instance one can investigate root-finding algorithms in the model $\cT_{\alpha, m}(n)$, where each arriving new node sends $m$ independent edges according to the same rule than in $\cT_{\alpha}(n)$.
\item If one has to output only a single vertex, what is the best possible success probability? For instance our results imply a lower bound of at least $1\%$ success probability in the uniform attachment model (the strategy is to pick at random a vertex in $H_{\psi}$ with $K=35$).
\end{enumerate}

\section*{Acknowledgements}

We thank Justin Khim and Po-Ling Loh for correcting a mistake in the original proof of Theorem \ref{th:psipa}. We also thank Mikl\'os Z.\ R\'acz for helpful discussions.

\section*{Appendix}
\begin{lemma}\label{lem:lucisaslave}
Let $E_1, E_2, \hdots$ be an i.i.d. sequence of exponential random variables (of parameter $1$), and let 
$$X = \prod_{i=1}^{+\infty} \min\left(\sum_{k=1}^i E_k, 1\right) .$$
Then for any $t >0$,
$$\P(X \leq t) \leq 6 t^{1/4} .$$
\end{lemma}

\begin{proof}
First observe that, almost surely, $X$ is a finite product, and thus one can write for $t \in (0,1)$,
\begin{eqnarray}
\P(X \leq t) & = & \P\left(\exists \ \ell : \sum_{k=1}^{\ell+1} E_k > 1, \ \text{and} \  \sum_{k=1}^{\ell} E_k \leq 1, \ \text{and} \ \prod_{i=1}^{\ell} \left(\sum_{k=1}^i E_k\right) \leq t \right) \notag \\
& \leq & \sum_{\ell=1}^{+\infty} \min\left(\P\left(\sum_{k=1}^{\ell} E_k \leq 1\right) ,  \P\left(\sum_{k=1}^{\ell+1} E_k > 1, \ \text{and} \  \ \prod_{i=1}^{\ell} \left(\sum_{k=1}^i E_k\right) \leq t \right) \right) . \label{eq:s1}
\end{eqnarray}
One has
$$\P\left(\sum_{k=1}^{\ell} E_k \leq 1\right) = \int_0^1 \frac{x^{\ell-1} \exp(-x)}{(\ell-1)!} dx \leq \frac{1}{\ell!}.$$
Furthermore we show below that 
\begin{equation} \label{eq:s2}
\P\left(\sum_{k=1}^{\ell+1} E_k > 1, \ \text{and} \  \ \prod_{i=1}^{\ell} \left(\sum_{k=1}^i E_k\right) \leq t \right) \leq 2^\ell \sqrt{t} .
\end{equation}
The two above displays together with \eqref{eq:s1} concludes the proof, since they imply (with $\min(a,b) \leq \sqrt{ab}$):
$$\P(X \leq t) \leq \sum_{\ell=1}^{+\infty} \sqrt{\frac{2^{\ell} \sqrt{t}}{\ell!}} \leq 6 t^{1/4} .$$ 
To prove \eqref{eq:s2} we first write
$$\prod_{i=1}^{\ell} \left(\sum_{k=1}^i E_k\right) = \left(\sum_{k=1}^{\ell+1} E_k \right)^{\ell} \prod_{i=1}^{\ell} \left(\frac{\sum_{k=1}^i E_k}{\sum_{k=1}^{\ell+1} E_k}\right). $$
Now observe that the vector $\left({\sum_{k=1}^i E_k} / {\sum_{k=1}^{\ell+1} E_k}\right)_{i=1,\hdots,\ell}$ is equal in law to $(U_{(i)})_{i=1,\hdots, \ell}$ where $U_1, \hdots, U_{\ell}$ is an i.i.d. sequence of uniform random variables in $[0,1]$ and $(U_{(i)})$ is an increasing rearrangement of $(U_i)$. Thus one has
$$\P\left(\sum_{k=1}^{\ell+1} E_k > 1, \ \text{and} \  \ \prod_{i=1}^{\ell} \left(\sum_{k=1}^i E_k\right) \leq t \right) \leq \P\left(\prod_{i=1}^{\ell} U_i \leq t\right) \leq \E \left(\frac{\sqrt{t}}{\sqrt{\prod_{i=1}^{\ell} U_i}}\right) \leq  2^{\ell} \sqrt{t} ,$$
which concludes the proof.
\end{proof}

\bibliographystyle{plainnat}
\bibliography{bib}

\end{document}